\documentclass[12pt, a4paper, leqno]{amsart}

\usepackage{amsmath} 
\usepackage{amsfonts}
\usepackage{amssymb}
\usepackage{txfonts}   
\usepackage{amsfonts}  
\usepackage{comment}
\usepackage{ae}   
\usepackage{graphicx}   
\usepackage{pst-all} 
\usepackage{xcolor} 
 
\newcommand{\R}{\varmathbb{R}}
\newcommand{\Z}{\varmathbb{Z}}
\newcommand{\Rn}{{\varmathbb{R}^n}}

\newcommand{\ve}{\varepsilon}

\newcommand{\vint}{\operatornamewithlimits{\boldsymbol{--} \!\!\!\!\!\! \int }}

\def\diam{\qopname\relax o{diam}}
\def\loc{\qopname\relax o{loc}}

\def\dist{\qopname\relax o{dist}}
\def\min{\qopname\relax o{min}}
\def\diam{\qopname\relax o{diam}}

\usepackage{amsthm}

\swapnumbers
\theoremstyle{plain}
\newtheorem{theorem}[equation]{Theorem}
\newtheorem{lemma}[equation]{Lemma}

\newtheorem{corollary}[equation]{Corollary}

\theoremstyle{definition}
\newtheorem{definition}[equation]{Definition}

\newtheorem{example}[equation]{Example}

\theoremstyle{remark}
\newtheorem{remark}[equation]{Remark}

\numberwithin{equation}{section}

\title{Pointwise estimates to the modified Riesz potential }

\author{Petteri Harjulehto}
\address[Petteri Harjulehto]{Department of Mathematics and Statistics,
FI-20014 University of Turku, Finland}
\email{petteri.harjulehto@utu.fi}

\author{Ritva Hurri-Syrj\"anen}
\address[Ritva Hurri-Syrj\"anen]{Department of Mathematics and Statistics,
FI-00014 University of Helsinki, Finland}
\email{ritva.hurri-syrjanen@helsinki.fi}

\date{\today}

\begin{document}


\begin{abstract} 
We prove pointwise estimates to the modified Riesz potential. We show the boundedness of its Luxemburg norm.
As an application we obtain Orlicz embedding results for $L^1_p$-functions, $1\le p<n$.  We study the sharpness of the results. 
\end{abstract}

\maketitle

\markboth{\textsc{Petteri Harjulehto and Ritva Hurri-Syrj\"anen }}
{\textsc{Petteri Harjulehto and Ritva Hurri-Syrj\"anen }}

\section{Introduction}

We study pointwise estimates and integral estimates for the modified Riesz potential
\begin{equation}\label{potential}
\int_{G}  \frac{|f(y)|}{\varphi(  |x-y|)^{n-1}} \, dy
\end{equation}
where $\varphi$ is a continuous, strictly increasing
function on $[0,\infty )$
such that for some constant $C_\varphi$
\begin{equation}\label{varphi_control}
\frac{\varphi (t_1)}{t_1}\le C_\varphi
\frac{\varphi (t_2)}{t_2}
\quad \text{whenever}\quad  0<t_1\le t_2\,,
\end{equation}
and $f$ is a locally integrable function defined on an open set $G$ in the Euclidean  $n$-space  $\Rn$, $n\geq 2$.
We give sufficient conditions to an Orlicz function $H$  such that the inequality
\begin{equation*}
H\left ( \int_{G}  \frac{|f(y)|}{\varphi(  |x-y|)^{n-1}} \, dy\right )
\le C (M f (x))^p
\end{equation*}
holds for every $x\in \Rn$, whenever
$\| f \|_{L^p (\Rn)} \le 1$, and the constant
$C$ is independent of $f$.
Here, $Mf$ is
the Hardy-Littlewood maximal function for balls,
\[
Mf(x) = \sup_{r>0} \vint_{B(x,r)} |f(y)| \, dy\,.
\]
This inequality implies the boundedness of the corresponding Luxemburg norm.
As an application we obtain Orlicz embedding results
for $L^1_p$-functions which are defined on domains with fractal boundaries, whenever $1\le p< n$.

Our main theorems are 
Theorem \ref{thm:pointwise-maximal},
the pointwise estimate
with the classical Hardy-Littlewood maximal function, and
Theorem \ref{thm:Riesz-is-bounded},
the boundedness of the Luxemburg norm of the modified Riesz potential.

\begin{theorem}\label{thm:pointwise-maximal}
Let $1\le p<n$ be given.
Let $\varphi : [0,\infty )\to [0, \infty)$ be a continuous, strictly increasing function which satisfies
condition  \eqref{varphi_control}.
Suppose that
there exists a continuous function $h:[0,\infty )\to [0, \infty)$ so that 
\begin{equation}\label{h_sum}
\sum_{k=1}^{\infty} \frac{(2^{-k}t )^n}{\varphi (t 2^{-k})^{n-1}}
\le h(t ) \quad \text{ for all }\quad t>0\,.
\end{equation}
Let $\delta :(0,\infty )\to [0,\infty )$ be a continuous function and
let $H: [0,\infty )\to [0, \infty)$ be an Orlicz function satisfying the $\Delta _2$-condition. Suppose that
there exists a finite constant $C_H$ such that the inequality 
\begin{equation}\label{sum}
H\left(h(\delta (t)) t + 
\varphi (\delta (t) )^{1-n}(\delta(t)) ^{n(1-\frac{1}{p})} \right)\le
C_H t^p 
\end{equation}
holds for all $t>0$.
Let $G$ in $\Rn$ be an open set.
If $\| f \|_{L^p (\Rn)} \le 1$, then 
there exists  a constant $C$ such that 
the inequality
\begin{equation}\label{main_riesz_inequality}
H\left ( \int_{G}  \frac{|f(y)|}{\varphi(  |x-y|)^{n-1}} \, dy\right )
\le C (M f (x))^p
\end{equation}
holds for every $x\in \Rn$. Here the constant $C$ depends on $n$, $p$, $C_\varphi$, $C_H$, and the $\Delta_2$-constant of $H$ only.
\end{theorem}

Theorem \ref{thm:pointwise-maximal} implies the boundedness of the modified Riesz potential.

\begin{theorem}\label{thm:Riesz-is-bounded} 
Let $H$ 
be an Orlicz function and $\varphi$ be an increasing function as in Theorem \ref{thm:pointwise-maximal}.
Let $G$ be an open set in $\Rn$.
Then there exists a constant $C$ such that 
the inequality
\[
\int_G H \left ( \int_{G}  \frac{|f(y)|}{\varphi(  |x-y|)^{n-1}} \, dy\right ) \, dx \le C
\]
holds for every  $f$ when  $\|f\|_{L\log L (G)}\le 1$ if $p=1$ in \eqref{sum},  and for every
$f$  when $\|f\|_{L^p(G)}\le 1$ if  $1<p<n$ in 
\eqref{sum}. Here the constant $C$ depends on $n$, $p$, $C_\varphi$, $C_H$, and the $\Delta_2$-constant of $H$ only.
\end{theorem}

The integral estimates imply the corresponding norm
estimates and then the boundedness of the Luxemburg norm of the Orlicz function follows, see Corollary
\ref{cor:Riesz-is-bounded}.

\begin{remark}
\begin{enumerate}
\item
Theorem \ref{thm:pointwise-maximal} reduces to the classical pointwise estimate
for the Riesz potential $I_\alpha f$,
\begin{equation*}
I_\alpha f(x)=\int_{\Rn}  \frac{|f(y)|}{|x-y|^{n-\alpha}} \, dy\,,
\end{equation*}
that is, there exists a constant $C$ such that
\begin{equation*}
|I_{\alpha}f(x)|^{np/(n-\alpha p)}\le 
C Mf(x)^p\left\| f\right\|^{\alpha p np/n(n-\alpha p)}_{L^p(\Rn )}\,,
\end{equation*}
when $\alpha\in(0, 1]$ and 
$1< p< n$, 
\cite[(3) in the proof of Theorem 1]{Hed72}.
Indeed,
if $f\in L^p(\Rn )$ is given and
we choose
$\varphi (t)=t^{\frac{n-\alpha}{n-1}}$, $h(t)=t^\alpha$, $\delta (t)=\frac{t^{-\frac{p}{n}}}{\|f \|_{L^p(\Rn)}}$,
and
$H(t)=\frac{n-\alpha p}{np}t^{np/(n- \alpha p)}$, then
the assumptions of Theorem \ref{thm:pointwise-maximal}
are valid.    
If $\alpha \in (1,n)$, then inequality \eqref{varphi_control} fails and
we can not use the method of our proof 
for Theorem \ref{thm:pointwise-maximal}.
\item
The classical $(np/(n-\alpha p),p)$-inequality
for the Riesz potential $I_\alpha f$, that is,
for $\alpha >0$, $1<p<\infty$, and $\alpha p<n$
there is a constant $C(n,p,\alpha )$ such that
\begin{equation*}
\left\|I_{\alpha}f\right\|
_{L^{np/(n-\alpha p)}(\Rn)}
\le C(n,p,\alpha )\left\|f\right\|_{L^{p}(\Rn )}
\end{equation*} 
whenever $f\in L^p(\Rn )$,
\cite[Theorem 1]{Hed72}, 
is a special case of  
Theorem \ref{thm:Riesz-is-bounded} with $\varphi(t) =t^{\frac{n-\alpha}{n-1}}$ and $H(t) =\frac{n-\alpha p}{np} t^{np/(n-\alpha p)}$ whenever $\alpha\in(0, 1]$ and $1<p<n$.
\item
Trudinger's inequality 
\cite[p. 479]{Tru},   \cite[Theorem 1]{Moser1971},    \cite[Theorem]{Strichartz1972}, and
\cite[Theorem 2]{Hed72} 
for functions with compact support follows from 
Theorem \ref{thm:Riesz-is-bounded} 
when $\varphi (t)=t$
as in 
\cite[p. 507]{Hed72}.
\end{enumerate}
\end{remark}

More generally
boundedness results to the Riesz operator $I_{ \alpha }$ from an Orlicz space to another Orlicz space are found in \cite {ONeil1965},  \cite {Torchinsky1976}, 
\cite {Kokilashvili-Krbec}, 
and 
 \cite{Cianchi}. Andrea Cianchi  characterized the Orlicz functions which give the corresponding norm inequalities,
 \cite[Theorem 2 (ii)]{Cianchi1999}. 
Cianchi and Bianca Stroffolini gave simplified proofs, 
\cite[Theorem 1 , Corollary 1]{Cianchi_Stroffolini}. For recent developments we refer to  \cite{Ohno_Shimomura}.

We are interested in the modified Ries potential \eqref{potential}
which is  the classical Riesz potential whenever
$\varphi (t)=t^{\frac{n-\alpha}{n-1}}$.
We prove a  norm estimate for
the Luxemburg norm of the modified Riesz potential under certain assumptions on $\varphi $
whenever functions are $ L^1_p$-functions, Corollary
\ref{cor:Riesz-is-bounded}.

As an application
we obtain Orlicz embedding results for $L^1_p$-functions,
which are defined on 
bounded domains with a cone condition
but also if the functions are 
defined on more irregular domains, see Section 
\ref{proof} and
Theorems \ref{thm:main_application}
and
\ref{thm:main_application_b} there.

The definitions which we need are recalled in Section
\ref{preliminaries}.
After we have collected some auxiliary results
in Section \ref{estimates}
we prove the pointwise result Theorem \ref{thm:pointwise-maximal}
and the boundedness results Theorem \ref{thm:Riesz-is-bounded} and Corollary \ref{cor:Riesz-is-bounded} .
The properties of domains with fractal boundaries
are studied in Section \ref{domains}.
We state and prove embedding results 
in Section \ref{proof}.
Examples of functions to which our results apply are given in Section \ref{corollary_and_examples}.
We give sharpness results in Section \ref{sec:example}.

\section{Preliminaries}
\label{preliminaries}

Throughout this paper we assume that
the function $H:[0, \infty) \to [0, \infty)$ has the properties
\begin{enumerate}
\item $H$ is continuous,
\item $H$ is strictly increasing,
\item $H$ is convex,
\item $\lim_{t \to 0^+}\frac{H(t)}{t} =0$ and $\lim_{t \to \infty}\frac{H(t)}{t} =\infty$,
\item $\frac{H(t)}{t} < \frac{H(s)}{s}$ for $0<t<s$.
\end{enumerate}
In other words, we suppose that $H$ is an $N$-function,
\cite[8.2]{AF}. We also assume that $H$ satisfies the $\Delta_2$-condition, that is, there exists a constant
$C^{\Delta_2}_H$ such that
\begin{equation}\label{H_doubling}
H (2t)\le C^{\Delta_2}_H H(t) \quad \text{for all} \quad t>0.
\end{equation}
If the function satisfies the $\Delta_2$-condition, we say that the function is a $\Delta_2$-function.

Let $G$ in $\Rn$ be an open set.
The Orlicz class is a set of all measurable functions
$u$
defined on $G$
such that
\[
\int_G H \Big(|u(x)| \Big) \, dx < \infty\,.
\]
We study the Orlicz space $L^H (G)$ which means 
a space of all measurable functions
$u$ defined on $G$ such that 
\[
\int_G H \Big(\lambda |u(x)| \Big) \, dx < \infty
\]
for some  $\lambda >0$.

When the function $H$
satisfies the $\Delta_2$-condition, then the space $L^H (G)$ is a vector space and it is 
equivalent to the corresponding Orlicz class.
We study these Orlicz spaces and call their functions Orlicz functions.
The Orlicz space
$L^H (G)$ equipped with the Luxemburg norm
\[
\|u\|_{L^\Phi(G)} = \inf \left\{\lambda >0: \int_G \Phi\left ( \frac{|u(x)|}{\lambda}\right)\, dx\le 1 \right\}
\]
is a Banach space.

Let  $X$ and $Y$ be normed  spaces. 
Then,
$X$ is continuously embedded in $Y$,
that is, there exists a continuous embedding from $X$ into $Y$, 
written as
$X \hookrightarrow Y$, if
there exists a constant $C$ such that
\[
\|u\|_Y\le C\|u\|_X
\]
for all $u\in X$.

We note that if the Lebesque measure of $G$ is finite, then there is a continuous embedding 
$L^H(G) \hookrightarrow L^1(G)$.
Examples of Orlicz spaces are 
$L^p(G):=L^{H(G)}$ when $H(t)=t^p$, $1<p<\infty$, 
and $L\log L(G):=L^H(G)$ when $H(t)=t\log(e+t)$.
For more information about Orlicz spaces we refer to \cite[Section 8]{AF} and
\cite[Section 6.3]{G}.

The space of locally integrable functions defined on an open set $G$ is written as $L^1_{\loc}(G)$.
We recall that the space
$ L^1_p(G)\,,$
$1\le p <\infty$, is a space of distributions on $G$ with the first order derivatives in the space $L^p(G)$.

An open ball with a center $x$ and radius $r>0$ is written as
$B(x,r)$. The corresponding closed ball is denoted by $\overline{B}(x,r)$.
Given any proper subset $A$ of $\Rn$ and any $x\in\Rn$, 
the distance between $x$ and the boundary $\partial A$ is written as
$\dist (x,\partial A)$, and $\diam (A)$ stands for the diameter of $A$.
The characteristic function of a set $A$ is denoted by $\chi_A$.
When $A$ in $\Rn$  is a Lebesgue measurable set with positive $n$-Lebesgue measure $|A|$ we write 
the integral average of 
an integrable function $u$ in $A$ as
\begin{equation*}
u_A=\vint_A u(x)\,dx=\vert A\vert ^{-1}\int_A u(x)\,dx\,.
\end{equation*}

We let $C(*,\cdots ,*)$ denote a constant which depends on the quantities appearing in the parentheses only.
In the calculations from one line to the next line we usually write $C$ for constants when it is not important to specify 
constants'
dependence on the quantities appearing in the calculations. From line to line $C$ might stand for a different constant.

\section{Pointwise estimates  for a modified Riesz potential}\label{estimates}

The classical Hardy-Littlewood maximal function is written as 
\[
Mf(x) = \sup_{r>0} \vint_{B(x,r)} |f(y)| \, dy
\]
where
$f$ is a
locally integrable function defined on $\Rn$, \cite[Section 1]{S}.
We give two pointwise estimates by using the Hardy-Littlewood maximal operator in Lemmas \ref{lem:pointwise-maximal-1}
and \ref{lem:pointwise-maximal-2}.
Lars Inge Hedberg stated and proved the corresponding
results when $\varphi (t)=t$, \cite[Lemma (a), (b)]{Hed72}.  
Cianchi and Stroffolini used the Hedberg method 
for the classical Riesz potential when functions are Orlicz functions, 
\cite[Theorem 1 , Corollary 1]{Cianchi_Stroffolini}.

\begin{lemma}\label{lem:pointwise-maximal-1}
Let $\varphi : [0,\infty )\to [0,\infty )$ be
a continuous, strictly increasing function.
Let $h : [0,\infty )\to [0,\infty )$ be a continuous function with the condition \eqref{h_sum}.
Let  $\delta >0$ be given.
If $f \in L^1_{\loc}(\Rn)$, then there exists a constant $C(n)$ such that 
the inequality
\[
\int_{B(x, \delta)}  \frac{|f(y)|}{\varphi( \vert x-y\vert )^{n-1}}\, dy\le C(n)h(\delta )Mf(x)
\]
holds 
for every $x \in \Rn$.
\end{lemma}

\begin{proof}
Let $x\in\Rn$ be fixed and let $\delta$ be given.
Let us divide the ball $B(x, \delta)$ into annuli. By
bringing in the Hardy-Littlewood maximal operator and
by using inequality \eqref{h_sum} we obtain
\[
\begin{split}
\int_{B(x, \delta)}  \frac{|f(y)|}{\varphi(  |x-y|)^{n-1}} \, dy
&\le \sum_{k=1}^\infty \varphi( \delta 2^{-k})^{1-n} \int_{\{z:2^{-k} \delta \le |x-z|< 2^{-k +1} \delta\}} |f(y)| \, dy \\
&\le C(n) \sum_{k=1}^\infty \frac{(2^{-k} \delta)^n}{\varphi(\delta 2^{-k})^{n-1}}  \vint_{\{z: |x-z|< 2^{-k +1} \delta\}} |f(y)| \, dy \\
&\le C(n) Mf(x) \sum_{k=1}^\infty \frac{(2^{-k} \delta)^n}{\varphi(\delta 2^{-k})^{n-1}} \le C(n)Mf(x)h(\delta ). \qedhere
\end{split}
\]
\end{proof}

We consider the integral over the set $\Rn\backslash B(x,\delta )$, too.

\begin{lemma}\label{lem:pointwise-maximal-2}
Let  $\varphi : [0,\infty )\to [0,\infty )$ be
a continuous strictly increasing function such that
inequality \eqref {varphi_control} holds.
Let $1\le p<n$. Let $\delta >0$  be given.
If $\|f\|_{L^p(\Rn)} \le 1$, then there is a constant $C$, depending on $n$, $p$, and $C_\varphi$ only  such that 
the inequality
\[
\int_{\Rn\setminus B(x, \delta)}  \frac{|f(y)|}{\varphi(  |x-y|)^{n-1}} \, dy 
\le C  \varphi(\delta)^{1-n}\delta^{n(1-\frac{1}{p})}
\]
holds
for every $x\in \Rn$.
\end{lemma}

The inequality in Lemma \ref{lem:pointwise-maximal-2} has been proved for the function
$\varphi (t)=t^{\alpha}/\log^{\beta} (e+t^{-1})$ when
$1\le\alpha < 1+1/(n-1))$ and $\beta\geq 0$ in \cite[Lemma 3.2]{HH-S}. 
The proof here is a generalization of this earlier result. We give the proof for the sake of completeness.

\begin{proof}[Proof of Lemma \ref{lem:pointwise-maximal-2}]
Suppose that $1<p<n$ and
let us write $p'=p/(p-1)$. Let the point $x\in \Rn$ be 
fixed  and let $\delta >0$ be given.
By H\"older's inequality we obtain
\[
\begin{split}
\int_{\Rn\setminus B(x, \delta)}  \frac{|f(y)|}{\varphi(|x-y|)^{n-1}} \, dy  
&\le \|f\|_{L^{p}(\Rn)} \big\|\chi_{\Rn\setminus B(x, \delta)} \varphi( |x- \cdot |)^{1-n} \big\|_{L^{p'}(\Rn)}\\
&\le \big\|\chi_{\Rn\setminus B(x, \delta)} \varphi(|x- \cdot |)^{-n} \big\|
^{(n-1)/n}
_{L^{(n-1)p'/n}(\Rn)}.
\end{split}
\]
For every $y\in \Rn\setminus B(x, \delta)$,
\[
\begin{split}
\varphi( |x- y |)^{-n} & = C(n) |B(y, \varphi( |x- y |))|^{-1}\\
&=  C(n) \vint_{B(y,  2 |x- y |)} \chi_{B(x, \delta)}(z) |B(x, \delta)|^{-1} \frac{|B(y,  2 |x- y |)|}{|B(y, \varphi( |x- y |))|}  \, dz.\\
\end{split}
\]
By assumption \eqref{varphi_control} we obtain that 
\[
\frac{|B(y,  2 t)|}{|B(y, \varphi( t))|}  \le  C(n, C_\varphi) \left(\frac{\delta}{\varphi(\delta)} \right)^n
\]
for every $t \ge \delta$.   Hence,
\[
\varphi( |x- y |)^{-n} \le C(n, C_\varphi) \left(\frac{\delta}{\varphi(\delta)} \right)^n 
 M \Big(\chi_{B(x, \delta)} |B(x,   \delta)|^{-1}\Big)(y)
\]
for every $y\in \Rn\setminus B(x, \delta)$.

Since $1<p<n$,
we have  $1<\frac{n-1}{n} p'<\infty$. Thus, the Hardy-Littlewood maximal operator is bounded 
in 
$L^{(n-1)p'/n}(\Rn )$, \cite[Section 1, Theorem 1(c)]{S}, and we obtain
\[
\begin{split}
&\big\|\chi_{\Rn\setminus B(x, \delta)} \varphi(  |x- \cdot |)^{-n} \big\|^{(n-1)/n}_{L^{(n-1)p'/n}(\Rn)}\\ 
&\qquad \le C(n, C_\varphi) \left(\frac{\delta}{\varphi(\delta)} \right)^{n-1}
\left \|M \Big(\chi_{B(x, \delta)} |B(x, \delta)|^{-1}\Big) \right \|^{(n-1)/n}_{L^{(n-1)p'/n}(\Rn)} \\
& \qquad\le  C(n,C_\varphi,p) \left(\frac{\delta}{\varphi(\delta)} \right)^{n-1}\left \|\chi_{B(x, \delta)} |B(x, \delta)|^{-1} \right \|^{(n-1)/n}_{L^{(n-1)p'/n}(\Rn)}\\
& \qquad\le  C(n,C_\varphi,p) \varphi(\delta)^{1-n}\left \|\chi_{B(x, \delta)} \right \|^{(n-1)/n}_{L^{(n-1)p'/n}(\Rn)}\\
&\qquad\le C(n,C_\varphi, p) \varphi(  \delta)^{1-n} \delta^{\frac{n}{p'}}.
\end{split}
\]
Hence, the claim is proved whenever $1<p<\infty$.

If $p=1$, and
$\delta >0$  is given, and
$\|f\|_{L^1(\Rn)} \le 1$, then  
\[
\int_{\Rn\setminus B(x, \delta)}  \frac{|f(y)|}{\varphi(|x-y|)^{n-1}} \, dy 
\le {\varphi(\delta)^{1-n}} \int_{\Rn\setminus B(x, \delta)} |f(y)| \, dy
\le \varphi(\delta)^{1-n}. \qedhere
\]
\end{proof}

\begin{remark}\label{rem:pointwise-maximal-2}
Let us assume that $\varphi(t) = \frac{t^{\alpha}}{\log ^{\beta }(e+t^{-1})}$, $\alpha \in \left[1, 1+ \frac{1}{n-1} \right)$ and $\beta \geq 0$. We use this $\varphi$ in 
 Corollary~\ref{main_corollary}. 
 In this case the restriction $p<n$ in Lemma~\ref{lem:pointwise-maximal-2}
can be replaced  by the inequality  $p<n/(n-\alpha(n-1))$.  This yields that in Theorems \ref{thm:pointwise-maximal}, \ref{thm:Riesz-is-bounded} and \ref{thm:main_application}
the restriction $p<n$ can be replaced  by $p<n/(n-\alpha(n-1))$.
In the  proof of Lemma~\ref{lem:pointwise-maximal-2} we may estimate the term
$\big\|\chi_{\Rn\setminus B(x, \delta)} \varphi( |x- \cdot |)^{1-n} \big\|_{L^{p'}(\Rn)}$ by using the following calculation:
\[
\begin{split}
&\big\|\chi_{\Rn\setminus B(x, \delta)} \varphi( |x- y |)^{1-n} \big\|_{L^{p'}(\Rn)}
= \left ( \int_{\Rn\setminus B(x, \delta)} \varphi( |x- \cdot |)^{p'(1-n)} \, dy \right )^{\frac1{p'}}\\
&\quad = \left ( C(n) \int_{\delta}^\infty \varphi( t)^{p'(1-n)} t^{n-1} \, dt \right )^{\frac1{p'}}\\
&\quad= C(n, p) \left ( \int_{\delta}^\infty  t^{\alpha p'(1-n)+n-1} \log^{\beta p' (n-1) }(e+t^{-1})  \, dt \right )^{\frac1{p'}}\\
&\quad\le  C(n, p) \log ^{\beta (n-1) }(e+\delta^{-1}) \left ( \int_{\delta}^\infty  t^{\alpha p'(1-n)+n-1}   \, dt \right )^{\frac1{p'}}.
\end{split}
\]
The last integral is finite if $\alpha p'(1-n)+n <0$. In this case we obtain
\[
\begin{split}
\big\|\chi_{\Rn\setminus B(x, \delta)} \varphi( |x- y |)^{1-n} \big\|_{L^{p'}(\Rn)}
&\le C(n, p) \log ^{\beta (n-1) }(e+\delta^{-1}) \delta^{\alpha(1-n)+\frac{n}{p'}} \\
&= \varphi(\delta)^{1-n} \delta^{n(1-\frac1p)}.
\end{split}
\]
\end{remark}

Now we are ready for the proofs of the main theorems.

\begin{proof}[Proof of Theorem \ref{thm:pointwise-maximal}]
We may assume that $Mf(x)>0$, since otherwise $f(x)=0$ almost everywhere.
By Lemmas~\ref{lem:pointwise-maximal-1} and \ref{lem:pointwise-maximal-2} 
there exists a constant $C$ such that
we obtain 
\begin{equation*}\label{main_riesz_inequality_1}
\begin{split}
&\int_{G}  \frac{|f(y)|}{\varphi(  |x-y|)^{n-1}} \, dy \le 
\int_{\Rn}  \frac{|f(y)|}{\varphi(  |x-y|)^{n-1}} \, dy\\
&\quad\le C h(\delta(Mf(x)) ) Mf(x)
+ C  \varphi( \delta(Mf(x)))^{1-n}
(Mf(x))^{n(1-\frac{1}{p})}\,
\end{split}
\end{equation*}
for every $x$ in $\Rn$.
Condition \eqref{sum} implies
for all $x$ in $\Rn$  
\[
H \left ( \int_{D}  \frac{|f(y)|}{\varphi( |x-y|)^{n-1}} \, dy \right) \le C (Mf(x))^p. \qedhere
\]
\end{proof}

This pointwise estimate gives the boundedness of the Luxemburg norm of the modified Riesz potential.

\begin{proof}[Proof for Theorem \ref{thm:Riesz-is-bounded}]
Suppose that $1<p<n$.
Let us assume that  $\| f \|_{L^p(G)} \le 1$.  Then by 
Theorem~\ref{thm:pointwise-maximal} the inequality
\begin{equation*}
H\left ( \int_{G}  \frac{|f(y)|}{\varphi(  |x-y|)^{n-1}} \, dy\right )
\le C \left(M f (x)\right)^p
\end{equation*}
holds
for every $x \in G$. 
Since the Hardy-Littlewood maximal operator
$M: L^p \to L^p$ is bounded whenever $1<p<\infty$,
we obtain by
integrating both sides of this inequality 
over $G$ 
\[
\begin{split}
\int_G H \left ( \int_{G}  \frac{|f(y)|}{\varphi(  |x-y|)^{n-1}} \, dy\right ) \, dx
&\le C \int_G \left(M f (x)\right)^p \, dx\\
&\le C \int_G | f (x) |^p \, dx \le C. 
\end{split}
\]

The proof in the case $p=1$  follows in the same way;
but the fact that the maximal operator $M:L \log L \to L^1$ is bounded had to be used instead of the boundedness of
the maximal operator $M:L^p \to L^p $ whenever  $1<p<\infty $.
\end{proof}

Note that if the inequality
\[
\int_G H \left ( \int_{G}  \frac{|f(y)|}{\varphi(  |x-y|)^{n-1}} \, dy\right ) \, dx \le C
\]
holds
for every  $f$ whenever $\|f\|_{L^p(G)}\le 1$,  where $1<p<\infty$, then
\begin{equation*}
\left \|\int_{G}  \frac{|f(y)|}{\varphi(    |x-y|)^{n-1}} \, dy \right\|_{L^H(G)} \le C
\end{equation*}
for every $f$ whenever  $\| f \|_{L^p(G)} \le 1$ and $1<p<\infty$.  The boundedness of the Luxemburg norm  
follows by applying this inequality 
to $f/\|f\|_{L^p(G)}$ whenever $1<p<\infty$.
Arguments in the case $p=1$ are similar.
We state the boundedness of the Luxemburg norm in the following corollary.

\begin{corollary}\label{cor:Riesz-is-bounded}
Let $H$ 
be an Orlicz function and $\varphi$ be an increasing function as in Theorem \ref{thm:pointwise-maximal}.
Let $G$ be an open set in $\Rn$.

If $1<p<n$, then
there exists a constant $C$ such that 
the inequality
\[
\left \|\int_{G}  \frac{|f(y)|}{\varphi( | \cdot -y|)^{n-1}} \, dy \right\|_{L^{H}(G)} \le C \|f\|_{L^p(G)}
\]
holds for every $f \in L^p(G)$. Here the constant $C$ depends on $n$, $p$, $C_\varphi$, $C_H$, and the $\Delta_2$-constant of $H$ only.

If $p=1$, then
there exists a constant $C_1$ such that 
the inequality
\[
\left \|\int_{G}  \frac{|f(y)|}{\varphi( | \cdot -y|)^{n-1}} \, dy \right\|_{L^{H}(G)} \le C_1 \|f\|_{L \log L (G)}
\]
holds for every $f \in L \log L(G)$. Here the constant $C_1$ depends on $n$, $C_\varphi$, $C_H$, and the $\Delta_2$-constant of $H$ only.
\end{corollary}

\section{Pointwise estimates for functions defined on  irregular domains}\label{domains}

We are going to give new embedding results for
$L^1_p$-functions, which are defined on domains with fractal boundaries. We recall the definition of
very irregular John domains and give an integral representation to functions defined on these domains.

\begin{definition}\label{john}
Let $\varphi : [0,\infty )\to [0,\infty )$ be a continuous, strictly increasing function.
A bounded domain $D$ in $\Rn\,, n\geq 2\,,$ is a $\varphi$-John domain if there exist a constant $c_J >0$ and a point $x_0 \in D$ such that each point $x\in D$ can be joined to $x_0$ by a rectifiable curve $\gamma:[0,l] \to D$, parametrized by its arc length, such that $\gamma(0) = x$, $\gamma(l) = x_0$, $l\leq c_J\,,$ and
\[
\varphi(t) \leq c_J \dist\big(\gamma(t), \partial D\big) \quad \text{for all} \quad t\in[0, c_J].
\]
The point $x_0$ is called a John center of $D$ and the constant $c_J$  is called a John constant of $D$.
\end{definition}
If a domain is a $\varphi$-John domain with a John center $x_0$, then it is a $\varphi$-John domain with any other $x \in D$,  but the John constant might be different.

Lipschitz domains, classical John domains, and the so called $s$-John domains
are examples of these domains.
But there are more irregular domains such as the mushrooms domain studied  in \cite[6. Example]{HH-S} and in \cite[6. Example]{HH-SK}.

The following lemma is needed to prove a pointwise integral representation to $L^1_1$-functions defined on a $\varphi$-John domain.
Lemma \ref{lem:balls} 
is a generalization of  \cite[Theorem 9.3]{Hajlasz-Koskela} 
where the classical John domain, corresponding to the case $\varphi (t)=t$, is considered. For the function
$\varphi (t)=t/\log (e+t^{-1})$ the corresponding result
has been proved in \cite[Lemma 3.5]{HH-SK}.
The definition of $\varphi$ affects to the property $(2)$ in Lemma \ref{lem:balls}.  The following inequality \eqref{varphi_bdd} is needed: 
There exists a constant $C'_\varphi$ depending on $\varphi$ and $c_J$ only such that
\begin{equation}\label{varphi_bdd}
\varphi (t)\le   C'_\varphi t \quad \text{for all}\quad  t \in [0, c_J].
\end{equation}
Namely, for a given John domain with a John constant $c_J$ by inequality \eqref{varphi_control} there exists  
a constant $C_\varphi$ such that 
\begin{equation*}
\varphi (t)\le  C_\varphi \frac{\varphi(c_J)}{c_J} t =: C'_\varphi t \quad \text{for all}   \quad  t \in [0, c_J].
\end{equation*}

\begin{lemma}\label{lem:balls}
Let $\varphi :[0,\infty )\to [0,\infty )$ be a continuous, strictly increasing $\Delta _2$-function satisfying 
inequality \eqref{varphi_bdd}.
 Let $D$ in $\Rn\,, n\geq 2\,,$ be a $\varphi$-John domain 
with a John constant $c_J$ and  a John center $x_0 \in D$. 
Then for every $x\in D\setminus B(x_0, \dist(x_0, \partial D))$ there exists a sequence of balls $\big(B(x_i, r_i)\big)$ such that $B(x_i, 2r_i)$ is in $D\,,$
$i=0,1,\dots\,,$ and
for some constants $K=K(c_J, C'_\varphi)$, $N=N(n)$, and $M=M(n)$ 
\begin{enumerate}
\item
$B_0 = B\Big(x_0, \frac12 \dist(x_0,  \partial D)\Big)$;
\item
$\varphi(\dist(x, B_i))\leq K r_i$, and $r_i \to 0 $ as $i\to \infty$;
\item
no point of the domain $D$ belongs to more than $N$ balls $B(x_i, r_i)$; and
\item
$|B(x_i, r_i) \cup B(x_{i+1}, r_{i+1})| \leq M |B(x_i, r_i) \cap B(x_{i+1}, r_{i+1})|$.
\end{enumerate}
\end{lemma}

\begin{proof}
Let $x\in D\setminus B(x_0, \dist(x_0, \partial D))$. 
Let $\gamma$ be a John curve joining $x$ to $x_0$,  its arc length written as
$l$.
 We write 
\[
B'_0 = B\Big(x_0, \frac14 \dist\big(x_0, \partial D \big)\Big)
\] 
and consider the balls $B'_0$ and
\[
B \Big(\gamma(t), \frac14 \dist\big(\gamma(t), \partial D \cup \{x\}\big)\Big),
\]
where $t\in(0,l)$.
By the Besicovitch covering theorem,  there is a sequence of closed balls 
\[
\overline{B'_1}, \overline{B'_2}, \ldots \quad\text{and}\quad \overline{B'_0}
\]
that cover the set $\{\gamma(t): t\in[0, l] \}\setminus \{x\}$ and have a uniformly bounded overlap depending on $n$ only. We  write $B(x_i,r_i)= 2B'_i$ for every
$i=0, 1\,, 2\,, \ldots$, where $x_i= \gamma(t_i)$,  $t_i\in (0,l)$,
and $r_i = \frac12 \dist\big(x_i, \partial D \cup \{x\}\big)$. 

By the fact that $\varphi$ is an increasing function and by the definition of  $\varphi$-John domain we obtain 
\begin{equation*}
 \varphi(\dist(x, B_0)) \le  \varphi (l)
\le  c_J \dist (x_0,\partial D)= 2 c_J r_0. 
\end{equation*}
Let us suppose then that $i\ge 1$.
If $r_i = \frac12 \dist(x_i, x)$, then
by inequality  \eqref{varphi_bdd} we obtain 
\[
\varphi (\dist(x,B(x_i,r_i)))\leq 
   C'_\varphi \dist(x,B(x_i,r_i)) \le 2 C'_\varphi r_i.
\]
If $r_i = \frac12 \dist(x_i, \partial D)$, then the fact that $\varphi$ is increasing and the definition of a $\varphi$-John domain give 
\[
\varphi (\dist(x,B(x_i,r_i)))\leq 
\varphi (\dist(x,x_i))\leq 
\varphi (t_i)\leq 
  c_J \dist(x_i,\partial D) =2 c_J r_i.
\]
Thus, property (2) holds.

We renumerate the balls $B(x_i,r_i)= 2B'_i$  
and leave out the extra balls.
Let $B_0$ be as before.
Assume that we have chosen balls $B_i$, $i=0,1, \ldots, m$. Then we choose the ball $B_{m+1}$ as follows. We trace along $\gamma$ starting from  $\gamma(t_m)$, which is the centre point of $B_m$, towards $x=\gamma(0)$ and choose the smallest $t_j$ for which $\gamma(t_j) = x_j \in B_{m}$. Note that the smallest $t_j$ exists. Let $\gamma(t')$ be the point where $\gamma$ leaves $B_m$ for the last time when we are going towards $\gamma(0) =x$. Then $\{\gamma(t): t\in[t', l]\}$ is covered finitely many balls $\overline{B'_i}$, since the balls have bounded overlapping and the radii have a uniform lower bound. 
Because of this changing of the picking order of the balls, we obtain that
$r_i \to 0$ and $x_i\to x$, whenever $i\to \infty$.

The point $x$ does not belong to any ball. Let $x'$ be any other point in the domain $D$. The point $x'$ cannot belong to the balls $B_i$ with $3r_i <\dist(x', x)$. If $x' \in B_i$, then 
\begin{equation*}
2 r_i \leq \dist(x, x_i) \leq \dist(x, x') + r_i.
\end{equation*}
Thus, we obtain that $x' \in B_i$ if and only if 
$$\frac13 \dist(x', x) \leq r_i \leq \dist(x, x').$$
The Besicovitch covering theorem implies that the balls with radius of $\frac14$ of the original balls are disjoint. Hence $x'$ belongs to less than or equal to
\[
N \frac{\vert B(x',2r)\vert}{\vert B(0,\frac1{12} r)\vert} = 24^n N
\]
balls $B_i$, where the constant $N$ is from the Besicovitch covering theorem and depends on the dimension $n$ only.
Hence, property (3) holds.

If $r_i = \frac12 \dist(x_i, \partial D)$ and $r_{i+1} = \frac12 \dist(x_{i+1}, \partial D)$, then $r_{i+1} \geq \frac12 r_i$ (since $x_{i+1} \in B_i$) and thus we obtain
\[
\frac{|B_i|}{|B_{i+1}|} \leq \left ( \frac{r_i}{\frac12 r_i}\right)^n = 2^n. 
\]
If $r_i = \frac12 \dist(x_i, x)$ and $r_{i+1} = \frac12 \dist(x_{i+1}, x)$, then $r_{i+1} \geq \frac12 r_i$ and thus we obtain
$|B_i|/|B_{i+1}| \leq 2^n$.   
If $r_i = \frac12 \dist(x_i, \partial D)$ and $r_{i+1} = \frac12 \dist(x_{i+1}, x)$, then  $r_i \leq \frac12 \dist(x_i, x)$ and we obtain the same ratio as before.
Similarly in the case when $r_i = \frac12 \dist(x_i, x)$ and  $r_{i+1} = \frac12 \dist(x_{i+1}, \partial D)$. We have shown
$|B_i| \leq 2^n |B_{i+1}|$. 
In the same manner we obtain $2r_{i+1} \le 3 r_i$ and hence $2^n|B_i| \geq  |B_{i+1}|$.  These yield property (4).
\end{proof}

The following pointwise integral representation
for $L^1_1$-functions defined 
on the classical John domain is well known,
\cite{R}, \cite{Haj01}.
The corresponding integral representation when
$\varphi (t)=t/\log (e+t^{-1})$ is proved in 
\cite[Theorem 3.4]{HH-SK}.
For the sake of completeness we give the proof for the general function $\varphi$ here.
Lemma \ref{lem:balls} is essential to this proof.

\begin{theorem}\label{thm:pointwise-John}
Let $\varphi :[0,\infty )\to [0,\infty )$ be a continuous,  strictly increasing $\Delta_2$-function  satisfying \eqref{varphi_bdd}.
Let $D$ in $\Rn\,, n\geq 2\,,$ be a $\varphi$-John domain
with a John constant $c_J$ and a John center  $x_0$. 
Then there exists a finite constant $C$ such that for every $u\in L^1_1(D)$ and for almost every $x\in D$ 
the inequality
\[
\big|u(x) - u_{B(x_0,  \dist(x_0, \partial D))}\big| \leq C \int_{D}  \frac{|\nabla u(y)|}{\varphi \big( |x-y|\big)^{n-1}} \, dy
\]
holds.
\end{theorem}

\begin{proof}
If $x\in  B(x_0, \dist(x_0, \partial D))$, then 
\[
\big|u(x) - u_{B(x_0,  \dist(x_0, \partial D))}\big| \leq \frac{\diam (B(x_0,  \dist(x_0, \partial D)))^n}{
n|B(x_0,  \dist(x_0, \partial D))|} \int_{
 B(x_0,  \dist(x_0, \partial D))}  \frac{|\nabla u (y)|}{ |x-y|^{n-1}} \, dy\,
\]
by \cite[Lemma 7.16]{Gilbarg-Trudinger}. 
Since by inequality \eqref{varphi_bdd} there is a constant $C'_\varphi$ such that
$\varphi(  |x-y|)^{n-1}  \le \big(C'_\varphi  |x-y| \big)^{n-1}$,  the claim
follows for points 
$x\in  B(x_0, \dist(x_0, \partial D))$.

Let us then assume that  $x\in D\setminus B(x_0, \dist(x_0, \partial D))$. Let 
$(B_i)_{i=0}^{\infty}$ be a sequence of balls constructed in Lemma~\ref{lem:balls}. Property (2) in Lemma \ref{lem:balls} gives that $\dist(x, B_i) \to 0$ whenever $i \to \infty$, since $\lim_{t\to 0+}\varphi (t)=0$ and $\varphi$ is continuously strictly increasing. Note that $\lim_{t\to 0+}\varphi (t)=0$ follows from the definition of $\varphi$-John domain by considering points near the boundary. Thus, property (2)  and the Lebesgue differentiation theorem \cite[Section 1, Corollary 1]{S} imply that 
\[
u_{B_i} \to u(x) \mbox{ when } i\to\infty
\] 
for almost every $x$. We obtain
\[
\begin{split}
|u(x) - u_{B_0}| &\leq \sum_{i=0}^\infty |u_{B_i} - u_{B_{i+1}}| \\
&\leq \sum_{i=0}^\infty \left (| u_{B_i} - u_{B_i \cap B_{i+1}}|  + |u_{B_{i+1}} - u_{B_i \cap B_{i+1}}| \right)\\
&\leq  \sum_{i=0}^\infty \left (\vint_{B_i \cap B_{i+1}}  |u(y)- u_{B_i} | \, dy  + \vint_{B_i \cap B_{i+1}} |u(y) - u_{B_{i+1}}|\, dy \right)\,.\\
\end{split}
\]
By property (4) in Lemma \ref{lem:balls}
\[
|u(x) - u_{B_0}| 
\leq  2C\sum_{i=0}^\infty \vint_{B_i}  |u(y)- u_{B_i} | \, dy. 
\]
By using the $(1,1)$-Poincar\'e inequality in a ball $B_i$, \cite[Section 7.8]{Gilbarg-Trudinger}, we obtain
\[
\begin{split}
|u(x) - u_{B_0}|
&\leq  C\sum_{i=0}^\infty r_i \vint_{B_i}  |\nabla u(y)| \, dy. \\
\end{split}
\]

By  \eqref{varphi_bdd} we have
$\varphi(2r_i) \le 2 C'_\varphi  r_i$. Since 
$\varphi$ is strictly increasing, there exists the strictly increasing inverse function $\varphi ^{-1}$ such that the inequality
$\varphi^{-1}(2C'_\varphi r_i)\ge 2 r_i$
holds. 
Thus, for each $z \in B_i$ we obtain by property (2) 
in Lemma \ref{lem:balls} that
\[
|x-z| \le \dist(x, B_i) + 2r_i \le \varphi^{-1}(K r_i) + 2r_i
\le  2 \varphi^{-1}\big(C r_i\big)\,,
\]
$C= \max\{K, 2 C'_\varphi\}$.
Hence, we have $ \varphi \big( \frac12 |x-z|\big) \le C r_i$.
By using this estimate and property (3) in Lemma \ref{lem:balls}
we obtain that
\[
\begin{split}
|u(x) - u_{B_0}|
&\leq  C\sum_{i=0}^\infty r_i \vint_{B_i}  |\nabla u(y)| \, dy 
\leq C\sum_{i=0}^\infty \int_{B_i}  \frac{|\nabla u(y)|}{r_i^{n-1}} \, dy \\
&\leq  C\sum_{i=0}^\infty  \int_{B_i}  \frac{|\nabla u(y)|}{\varphi \big( \frac12 |x-y|\big)^{n-1}} \, dy
\leq  C  \int_{D}  \frac{|\nabla u (y)|}{\varphi \big( \frac12 |x-y|\big)^{n-1}} \, dy.
\end{split}
\]
Since the function $\varphi$ safisfy the $\Delta_2$-condition, the claim follows.
\end{proof}

\section{Orlicz embbeding theorems}\label{proof}

Continuous embeddings into Orlicz spaces of exponential type for domains with a cone condition are well known,
\cite[Theorem 1, Theorem 2]{Tru}; we also refer to
\cite{Y}, \cite{P}, \cite{Peetre}. We recall that
Cianchi has proved sharp results for Orlicz-Sobolev spaces whenever relative isoperimetric inequalities are valid in the underlying domain, \cite[Theorem 2 and Example 1]{Cianchi}.
His work covers Orlicz spaces of exponential type and more. 
In particular, classical John domains, that is, 
$\varphi (t)=t$,
satisfy the Trudinger inequality, \cite[Example1]{Cianchi}.

We formulate the new embedding results for $L^1_p$-functions defined on $\varphi$-John domains.

\begin{theorem}\label{thm:main_application}
Let $\varphi : [0,\infty )\to [0, \infty)$ be a continuous, strictly increasing 
$\Delta_2$-
function which satisfies
condition  \eqref{varphi_control}. Let $H$ be an Orlicz function defined as in Theorem~\ref{thm:pointwise-maximal}.
Let $1<p<n$.
If $D$ in $\Rn\,, n\geq 2\,,$ is a $\varphi$-John domain
with a John center $x_0$,
then 
there exists a constant $C<\infty$ such that the inequality
\[
\int_D H(\vert u(x)-u_{B(x_0, \dist (x_0,\partial D)}\vert )\,dx\leq C 
\]
holds whenever $u\in L^1_p(D)$ and $\|\nabla u\|_{L^p(D)}\le 1$;
the constant $C$ does not depend on the function $u$.
\end{theorem}

If $p=1$, we need an extra assumption on the Orlicz function $H$.

\begin{theorem}\label{thm:main_application_b}
Let $\varphi : [0,\infty )\to [0, \infty)$ be a continuous, strictly increasing 
$\Delta_2$-
function which satisfies
condition  \eqref{varphi_control}. Let $H$ be an Orlicz function defined as in Theorem~\ref{thm:pointwise-maximal}. Let $D$ in $\Rn$, $n \ge 2$, be  a $\varphi$-John domain. 
If 
\begin{equation}\label{eq:H_summable}
\sum_{j=1}^\infty H(2^{-j})< \infty\,,
\end{equation}
then  there exists a finite constant $C$ such that the inequality
\[
\int_DH(\vert u(x)-u_{B(x_0, \dist(x_0, \partial D))}\vert )\,dx\leq C 
\] 
holds for every $u\in L^1_1(D)$ when 
$\| \nabla u\|_{L^1(D)}\le 1$;
the constant $C$ does not depend on the function $u$.
\end{theorem}

We state 
the corresponding norm inequalities next.

\begin{corollary}\label{cor:main_application}
If $1<p<n$,
let $\varphi$, $H$, and $D$ be defined as in Theorem~\ref{thm:main_application}. If  $p=1$, let 
$\varphi$, $H$, and $D$ be defined as in Theorem~\ref{thm:main_application_b}.
Then there exists a constant $C<\infty$ such that the inequality
\[
\|u - u_D \|_{L^{H}(D)} \leq C \| \nabla u\|_{L^p(D)}
\]
holds for every $u\in L^1_p(D)$;
the constant $C$ does not depend on the function $u$.
\end{corollary}

The pointwise estimate in Theorem \ref{thm:pointwise-John}
is crucial for the proofs.

\begin{proof}[Proof of the embedding result Theorem~\ref{thm:main_application}.]
Let $u\in L^1_p(D)$. 
Then, by \cite[1.1.2, Theorem]{Mazya} 
$u\in L_{\loc}^p(D)$.
Let $x_0 \in D$ be a John center of $D$. Theorem~\ref{thm:pointwise-John} and 
Theorem \ref{thm:Riesz-is-bounded} imply the claim.
\end{proof}

The proof of the embedding result is more tedious when $p=1$.

\begin{proof}[Proof of the embedding result Theorem~\ref{thm:main_application_b}.]
Let us consider functions $u \in L^1_1(D)$ such that
$\| \nabla u \|_{L^1(D)} \le 1$.  The center ball 
$B(x_0,\dist (x_0,\partial D ))$ is written as $B$.
We show  that there exists a constant $C<\infty$ such that the inequality
\begin{equation}\label{enough}
\int_D H(|u(x)-u_B|) \, dx \le C
\end{equation}
holds whenever $\|\nabla u\|_{L^1(D)} \le 1$.
First we estimate
\[
\int_D H(|u(x)-u_B|) \, dx \le \sum_{j \in \Z} \int_{ \{x\in D: 2^j < |u(x) - u_B| \le 2^{j+1}\}} H(2^{j+1}) \, dx.
\]
Let us define 
\[
v_j(x) = \max\bigg\{0, \min\Big\{|u(x) - u_B| - 2^j, 2^j\Big\}\bigg\}
\]
for all $x\in D$. 
\noindent
If $x \in \{x\in D :2^j < |u(x) - u_B|\le 2^{j+1}\}$, then $v_{j-1}(x) \ge 2^{j-1}$. 
We obtain
\begin{equation}\label{equ:main-2}
\int_D H(|u(x)- u_B|) \, dx \le \sum_{j \in \Z} \int_{\{x\in D :v_j(x) \ge 2^{j}\}} H(2^{j+2}) \, dx.
\end{equation}
By the triangle inequality and Theorem~\ref{thm:pointwise-John} we have
\[
\begin{split}
v_j(x) &= |v_j(x) - (v_j)_B + (v_j)_B| \le |v_j(x) - (v_j)_B| + |(v_j)_B|\\ 
&\le C \int_{D}  \frac{|\nabla v_j(y)|}{\varphi \big( |x-y|\big)^{n-1}} \, dy +  |(v_j)_B|
\end{split}
\]
for almost every $x\in D$.
By the $(1,1)$-Poincar\'e inequality in a ball $B$, \cite[Section 7.8]{Gilbarg-Trudinger}, there exists a constant $C$ such that
\[
|(v_j)_B|= (v_j)_B = \vint_B v_j(x) \, dx \le \vint_B |u(x) - u_B| \, dx  \le C \vint_B |\nabla u(x)| \, dx \le C |B|^{-1}.
\]
Thus, by the definition of $B$ the number $|(v_j)_B|$ is bounded by a constant depending  on $n$ and the distance between the John center and the boundary
of $D$ only.
We write
\[
I_{\varphi}(\nabla v_j)(x) = \int_{D}  \frac{|\nabla v_j(y)|}{\varphi \big( |x-y|\big)^{n-1}} \, dy.
\]
We continue to estimate the right hand side of inequality \eqref{equ:main-2}
\begin{equation}\label{equ:main-3}
\begin{split}
&\int_D H(|u(x)- u_B|) \, dx \le \sum_{j \in \Z} \int_{\{x\in D: C  I_{\varphi}(\nabla v_j)(x)+ C \ge 2^{j}\}} H(2^{j+2}) \, dx\\
&\quad \le \sum_{j \in \Z} \int_{\{x\in D : C  I_{\varphi}(\nabla v_j)(x) \ge 2^{j-1}\}} H(2^{j+2}) \, dx +\sum_{j=-\infty}^{j_0} \int_{D} H(2^{j+2}) \, dx.
\end{split}
\end{equation}
By \eqref{eq:H_summable} we obtain 
\begin{equation}\label{equ:main-4}
\sum_{j=-\infty}^{j_0} \int_{D} H(2^{j+2}) \, dx
= |D| \sum_{j=-\infty}^{j_0}H(2^{j+2}) \le C|D|.
\end{equation}
Then, we will find an upper bound for the sum
\begin{equation*}
\sum_{j \in \Z} \int_{\{x\in D : C  I_{\varphi}(\nabla v_j)(x) \ge 2^{j-1}\}} H(2^{j+2}) \, dx\,. 
\end{equation*}
Since $\|\nabla v_j\|_{L^1(D)} \le \|\nabla u\|_{L^1(D)} \le 1$, Theorem~\ref{thm:pointwise-maximal} implies that
\[
\begin{split}
\sum_{j \in \Z} \int_{\{x\in D: C  I_{\varphi}(\nabla v_j)(x) \ge 2^{j-1}\}} H(2^{j+2}) \, dx 
&\le \sum_{j \in \Z} \int_{\{x\in D: H(C I_{\varphi}(\nabla v_j)(x)) \ge H(2^{j-1})\}} H(2^{j+2}) \, dx \\ 
&\le  \sum_{j \in \Z} \int_{\{x\in D: C  M|\nabla v_j|(x) \ge H(2^{j-1})\}} H(2^{j+2}) \, dx.
\end{split}
\]
We choose for every $x \in \{x\in D: C  M|\nabla v_j|(x) \ge H(2^{j-2})\}$ a ball $B(x, r_x)$,
centered at $x$ and with radius $r_x$ depending on $x$, such that
\[
C \vint_{B(x,r_x)} |\nabla v_j (y)| \, dy \ge \frac12 H(2^{j-1})
\]
with the understanding that $|\nabla v_j|$ is zero outside $D$.
By the Besicovitch covering theorem
(or the $5$-covering theorem) we obtain a subcovering $\{B_k\}_{k=1}^{\infty}$ so that we may estimate by the $\Delta_2$-condition of $H$
\[
\begin{split}
&\sum_{j \in \Z} \int_{\{x\in D: C  I_{\varphi}(\nabla v_j)(x) \ge 2^{j-1}\}} H(2^{j+2}) \, dx
\le \sum_{j \in \Z} \sum_{k=1}^\infty  \int_{B_k}H(2^{j+2}) \, dx \le \sum_{j \in \Z} \sum_{k=1}^\infty  |B_k| H(2^{j+2})\\
&\quad\le \sum_{j \in \Z} \sum_{k=1}^\infty C |B_k|  \frac{H(2^{j+2})}{H(2^{j-1})} \vint_{B_k} |\nabla v_j (y)| \, dy
\le C \sum_{j \in \Z} \int_{D} |\nabla v_j (y)| \, dy.
\end{split}
\]
Let $E_j= \{x\in D :2^j < |u(x) - u_B|\le 2^{j+1}\}$.
Since $|\nabla v_j|$ is zero  almost everywhere in  $D\setminus E_j$ and
$|\nabla u (x)| = \sum_j |\nabla v_j (x)| \chi_{E_j}(x)$ for almost every $x \in D$, we obtain
\begin{equation}\label{equ:main-5}
\sum_{j \in \Z} \int_{\{x\in D: C  I_{\varphi}(\nabla v_j)(x) \ge 2^{j-1}\}} H(2^{j+2}) \, dx
\le C \int_{D} |\nabla u (y)| \, dy \le C.
\end{equation}
Estimates \eqref{equ:main-3}, \eqref{equ:main-4},
and \eqref{equ:main-5} imply inequality
\eqref{enough}.
\end{proof}

\begin{proof}[Proof of Corollary~\ref{cor:main_application}.]
Let us write $B = B(x_0, \dist(x_0, \partial D))$.
Theorem~\ref{thm:main_application}  for $1<p<n$ and   
Theorem~\ref{thm:main_application_b}  for $p=1$  yield 
 \[
\|u - u_B\|_{L^H(D)} \le C
 \]
for every $u\in L^1_p(D)$ with $\|\nabla u\|_{L^p(D)}\le 1$.
By using this inequality for $u/\|\nabla u\|_{L^p(D)}$ we obtain
 \[
\|u - u_B \|_{L^H(D)}  \le C  \| \nabla u\|_{L^p(D)}. 
\]
By the triangle inequality
\[
\|u- u_D\|_{L^H(D)} \le \|u- u_{B}\|_{L^H(D)} + \|u_{B}- u_D\|_{L^H(D)}.
\]
Here,
\[
\begin{split}
\|u_{B}- u_D\|_{L^H(D)} &= |u_{B}- u_D|
\, \|1 \|_{L^H(D)}\le
\|1 \|_{L^H(D)} \|u-u_{B}\|_{L^1(D)}\\
&\le C \|1 \|_{L^H(D)} \|u-u_{B}\|_{L^H(D)}
\end{split}
\]
for some constant $C$. The claim follows.
\end{proof}

\section{Examples}\label{corollary_and_examples}

As an application of  Theorems  \ref{thm:main_application} and \ref{thm:main_application_b} we obtain the following corollary. By Remark~\ref{rem:pointwise-maximal-2} we may replace the assumption $p<n$ by the assumption $p<n/(n-\alpha(n-1))$.

\begin{corollary}\label{main_corollary}
Let $1\le p<n/(n-\alpha(n-1))$ be given.
Let $\alpha \in [1, 1+1/(n-1))$ and $\beta \geq 0$.
Let
$\varphi : (0,\infty )\to \R$
and
$H : [0,\infty )\to \R$ 
be the functions 
\begin{equation*}
\varphi(t) = \frac{t^{\alpha}}{\log ^{\beta }(e+t^{-1})}
\end{equation*}
and
\begin{equation*}
H(t)=\biggl(
\frac{t}{\log ^{\beta (n-1)}(m+t)}
\biggr)^{\frac{np}{\alpha p(n-1)+n(1-p)}}
\end{equation*}
where $m=m(n,p)\ge e$.
If $D$ in $\Rn\,, n\geq 2\,,$ is a $\varphi$-John domain, then
there is a constant $C$ such that the inequality
\[
\int_DH(\vert u(x)-u_{B(x_0, \dist(x_0, \partial D))}\vert )\,dx \leq C
\]
holds for every $u\in L^1_p(D)$ when $\|\nabla u\|_{L^p(D)}\le 1$.
The constant $C$ does not depend on the function $u$.
\end{corollary}

Corollary~\ref{main_corollary} recovers the well known case
when $\alpha =1$ and $\beta =0$. But, 
Corollary~\ref{main_corollary} with $\alpha > 1$  and $\beta =0$ is not sharp. Namely, the exponent in the Orlicz function 
$H$
should be $\frac{np}{\alpha(n-1) - p +1}$,  and not ${\frac{np}{\alpha p(n-1)+ n(1 - p)}}$,
according to  \cite[p. 437]{Hajlasz-Koskela-2} and
\cite[Theorem 2.3]{KM}. We propose a conjecture that $\frac{np}{\alpha(n-1) - p +1}$ is the right exponent in the case $\beta >0$ also. We note that our method based on the modified Riesz potential does not give a better exponent than ${\frac{np}{\alpha p(n-1)+ n (1 - p)}}$, see Theorem~\ref{thm:sharpness-Riesz}.

Before the proof we point out that some earlier Orlicz embedding results are recovered.

\begin{example}
If we choose 
\begin{align*}
&\varphi (t)=\frac{t}{\log (e+t^{-1})}\,,\\
&\delta (t) = t^{-p/n}\,,\\
&h(t)= t \log^{n-1}(e+\delta^{-1})\,,\\
&H(t)=\frac{t^{np/(n-p)}}{(\log ^{n-1}(m+t))^{np/(n-p)}}\,, \mbox{ where } m\geq e,
\end{align*}
in Theorem  \ref{thm:main_application},
then
Theorem \ref{thm:main_application} recovers 
\cite[Theorem 1.1]{HH-SK}.
\end{example}

\begin{example}
If we choose 
\begin{align*}
&\varphi (t)=\frac{t^{\alpha}}{\log^{\beta} (e+t^{-1})}\,, 1\le\alpha < 1+1/(n-1))\,,
\\
&\delta (t) = t^{-p/n}\,,\\
&h(t)=t ^{n+(1-n)\alpha}\log^{\beta(n-1)}(e+\delta^{-1})\,,\\
&H(t)=
\biggl(\frac{t}{\log ^{\beta (n-1)}(m+t)}
\biggr)^{np/(\alpha p(n-1)+n-np)}\,,
\end{align*}
 in Theorem \ref{thm:main_application},
then
\cite[Theorem 4.1]{HH-S} follows.
\end{example}

We give a detailed proof Corollary
\ref{main_corollary}, since the proof shows why
the values of $\alpha$ should have
the upper bound $n/(n-1)$.
In Remark \ref{main_alpha_remark} we will point out that
the upper bound $n/(n-1)$ is the best possible with this Hedberg-type method for the modified Riesz potentials.

\begin{proof}[Proof of Corollary \ref{main_corollary}]
Let $\alpha \in [1, 1+1/(n-1))$ and $\beta \geq 0$. When
\begin{align*}
\varphi (t)=\frac{t^{\alpha}}{\log^{\beta} (e+t^{-1})}\,,
\end{align*}
calculations show that the
$\Delta_2$-condition of $\varphi$ and
inequality \eqref{varphi_control}  hold. 

In order to have condition \eqref{h_sum} we substitute  $\varphi $ to the left hand side of
\eqref{h_sum}  and estimate, for $\alpha < \frac{n}{n-1}$,
\[
\begin{split}
\sum_{k=1}^\infty \frac{\left(2^{-k}t \right)^n}{\varphi(2^{-k} t)^{n-1}}
&= t^{n-\alpha(n-1)} \sum_{k=1}^\infty 2^{-k(n-\alpha(n-1)} \log^{\beta(n-1)} \left(e + \frac{2^k}{t} \right)\\
&\le t^{n-\alpha(n-1)} \sum_{k=1}^\infty 2^{-k(n-\alpha(n-1)} \log^{\beta(n-1)} \left(2^k\left(e + \frac{1}{t}\right) \right) \\
&= t^{n-\alpha(n-1)} \sum_{k=1}^\infty 2^{-k(n-\alpha(n-1)} \left ( \log^{\beta(n-1)} \left(2^k \right) +  \log^{\beta(n-1)} \left(e + \frac{1}{t} \right)  \right)\\
&= C_1 t^{n-\alpha(n-1)} + C_2 t^{n-\alpha(n-1)}\log^{\beta(n-1)} \left(e + \frac{1}{t} \right) \\
& \le C t^{n-\alpha(n-1)}\log^{\beta(n-1)} \left(e + \frac{1}{t} \right)\,.
\end{split}
\]
Thus, we may choose
\[
h( t )= C t ^{n+(1-n)\alpha}\log^{\beta(n-1)}(e+t^{-1}).
\]
Let
\[
H(t)=
\biggl(\frac{t}{\log ^{\beta (n-1)}(m+t)}
\biggr)^{np/(\alpha p(n-1)+n-np)}\,.
\]
Conditions  \eqref{H_doubling} and \eqref{eq:H_summable} for the function $H$ hold
clearly.

We choose $\delta: (0, \infty) \to (0, \infty)$, $\delta(t) = t^{-\frac{p}{n}}$, and
show that condition \eqref{sum} holds with $\delta$, $h$, and $H$.
By substituting $h$ and  $\delta$  to the left hand side of \eqref{sum} we obtain
\[
\begin{split}
&H(h(\delta(t))t + \varphi(\delta (t))^{1-n}\delta (t)^{n(1-\frac{1}{p})}\\
&\quad = H\left( C t^{-p+p\alpha-\frac{\alpha p}{n}+1}  \log^{\beta(n-1)}(e+t^{\frac{p}{n}})  + t^{\alpha\frac{p}{n}(n-1)-p+1} \log^{\beta (n-1)}(e+t^{\frac{p}{n}}) \right)\\
&\quad = H\left(2 
C t^{-p+p\alpha-\frac{\alpha p}{n}+1}  \log^{\beta(n-1)}(e+t^{\frac{p}{n}})  
 \right)\\
 &\quad = H\left(2 
C t^{\frac{\alpha p(n-1)+n-np}{n}}  \log^{\beta(n-1)}(e+t^{\frac{p}{n}})  
 \right).
\end{split}
\]
The definition of $H$ and straightforward estimates imply
\[
\begin{split}
&H\left(
h(\delta(t))t + \varphi(\delta (t))^{1-n}\delta (t)^{n(1-\frac{1}{p}}
\right)\\
&\qquad \le  \frac{C t^{p} 
\left(\log^{\beta(n-1)}(e+t^{\frac{p}{n}})\right)^{\frac{pn}{\alpha p(n-1)+n-np}}}
{\left(\log^{\beta (n-1)} \left(m +  2C t^{\frac{\alpha p(n-1)+n-np}{n}}  
\log^{\beta(n-1)}(e+t^{\frac{p}{n}}) \right) \right)^{\frac{pn}{\alpha p(n-1)+n-np}}}\\
&\qquad\le C t^p \left(\frac{\log^{\beta(n-1)}(e+t^{\frac{p}{n}})}{\log^{\beta (n-1)} \left(m +  2C t^{\frac{\alpha p(n-1)+n-np}{n}}\right)} \right)^{\frac{pn}{\alpha p(n-1)+n-np}}\\
&\qquad\le C t^p. 
\end{split}
\]

Thus the claim follows by Theorems~\ref{thm:main_application} and \ref{thm:main_application_b}.
\end{proof}

\begin{remark}\label{main_alpha_remark}
We emphasize that the assumption $\alpha < 1 + \frac1{n-1}$ is natural when we consider the function
\begin{equation*}
\varphi (t)=\frac{t^{\alpha}}{\log^{\beta} (e+t^{-1})}\,.
\end{equation*}
Namely, if we assume that $\alpha \ge 1 + \frac1{n-1}$ and $\beta \ge 0$ and choose $f \equiv 1$ in $D$, then we obtain that
\[
\begin{split}
\int_{D}\frac{|f(y)|}{\varphi( \vert x-y\vert )^{n-1}}\, dy
&\ge \int_{B(x, \min\{1, \dist(x, \partial D)\})}\frac{\log^{\beta(n-1)}( e + |x-y|^{-1})}{\vert x-y\vert^{\alpha (n-1)}}\, dy\\
&\ge \int_{B(x, \min\{1, \dist(x, \partial D)\})}\frac{1}{\vert x-y\vert^{\alpha (n-1)}}\, dy\\
&\ge \int_{B(x, \min\{1, \dist(x, \partial D)\})}\frac{1}{\vert x-y\vert^{n}}\, dy\ = \infty
\end{split}
\]
for every $x\in D$.
\end{remark}

\section{Sharpness of the results}\label{sec:example}

In this section we study sharpness of the norm inequalities
\[
\left \| \int_D \frac{|u(z)|}{\varphi(|\cdot - z|)} \, dz \right \|_{L^H (D)} \le C \|u\|_{L^p(D)}
\]
and $\|u - u_D\|_{L^H(D)} \le C \|\nabla u\|_{L^p(D)}$. We start from the later inequality.

Let $\varphi : [0,\infty )\to [0, \infty)$ be a continuous, strictly increasing $\Delta_2$-function which satisfies
condition  \eqref{varphi_control}.
We give a sufficient condition to the function $H$
in Theorem \ref{ineq_fails}
so that the corresponding inequality in Theorem \ref{thm:main_application} fails. We do it
by constructing a mushrooms-type domain. 
Mushrooms-type domains can be found in \cite{Mazya-Poborchi}, \cite{Mazya}, \cite{HH-SV}, \cite{HH-SK}, \cite{HH-S}.
By using Theorem \ref{ineq_fails} we show that
the embedding in Theorem \ref{thm:main_application_b} is sharp.

Next we construct the mushrooms-type domain. Let $(r_m)$ be a decreasing sequence converging to zero. 
Let $Q_{m}$, $m=1,2,\dots$, be a closed cube in $\Rn$ with side length $2r_{m}$. Let $P_{m}$,
$m=1,2,\dots$,
be a closed rectangle in $\Rn$ which has side length $r_{m}$ for one side and $2\varphi (r_{m})$ for the remaining $n-1$ sides.  Let $Q_{0}=[0,1]^{n}$. We attach $Q_{m}$ and $P_{m}$ together creating 'mushrooms' which we then attach, as pairwise disjoint sets, to one side of  $Q_{0}$.  We have to assume here that $\varphi(r_m) \le r_m$. We attach the mushrooms to the side  that lies in the hyperplane $x_{2}=1$. 
We wish to define a domain that is symmetric with respect to the hyperplane $x_{2}=\frac12$. Thus, let $Q^{*}_{m}$ and $P^{*}_{m}$ be the images of the sets $Q_{m}$ and $P_{m}$, respectively, under the reflection across the hyperplane $x_{2}=\frac12$. We define
\begin{eqnarray}\label{eq:mushroon-domain}
G=\textrm{int}\left(Q_{0}\cup\bigcup_{m=1}^{\infty}\Big(Q_{m}\cup P_{m}\cup Q^{*}_{m}\cup P^{*}_{m}\Big)\right).
\end{eqnarray}

We give a sufficient condition to the Orlicz function $H$ so that the corresponding Orlicz embedding result in Theorem
\ref{thm:main_application} fails.

\begin{theorem}\label{ineq_fails}
Let $p \ge 1$. Let $\varphi : [0,\infty )\to [0, \infty)$ be a continuous, strictly increasing $\Delta_2$-function which satisfies
condition  \eqref{varphi_control}. Suppose that there exists $t_0>0$ such that $\varphi(t) \le t$ for $0<t<t_0$.
Let $G$ in $\Rn$, $n\geq 2$, be a mushrooms-type domain constructed as in \eqref{eq:mushroon-domain}.
If $H$ is an Orlicz function which satisfies the $\Delta_2$-property and the condition
\begin{equation*}
\lim_{t \to 0^+} t^nH\biggl(\biggl(\frac{t^{p-1}}{\varphi (t)^{n-1}}\biggr)^{1/p}\biggr)=\infty, 
\end{equation*}
then there exists a sequence of functions $(u_k)$ in $ L^1_p(G)$ such that  $\|\nabla u_k\|_{L^p(G)}=1$ for every $k$ and 
\begin{equation*}
\int_DH(\vert u_k(x)-(u_k)_D\vert )\,dx \to \infty \quad\text{as} \quad k \to \infty.
\end{equation*}
\end{theorem}

\begin{proof}
Let us define a sequence of piecewise linear continuous functions $(u_k)_{k=1}^{\infty}$ by setting
\begin{equation*}
u_{k}(x):=
\begin{cases}
F(r_k)& \textrm{in } Q_{k}, \\
-F(r_k) & \textrm{in } Q^{*}_{k},\\
0 & \textrm{in} Q_0,
\end{cases}
\end{equation*}
where the function $F$ will be  given in \eqref{F}.
Then the integral average of $u_{k}$ over $G$ is zero
for each $k$. 

The gradient of $u_k$ differs from zero in  $P_m \cup P_m^*$ only and 
\[
|\nabla u_k(x)| = \frac{F(r_m)}{r_m}, \textrm{ when }  x\in P_m \cup P_m^* \,.
\] 
Note that
\begin{equation*}
\int_G |\nabla u_k(x)|^p \, dx = 2\int_{P_m}
\biggl(\frac{F(r_m)}{r_m}\biggr)^p=
2r_m\left(\varphi(r_m)\right)^{n-1}\frac{F(r_m)^p}{r_m^p}\,.
\end{equation*}
We require that
\begin{equation*}
\int_G |\nabla u_k(x)|^p \, dx =1\,.
\end{equation*}
Hence,
\begin{equation}\label{F}
F(r_m)=\biggl(\frac{r_m^{p-1}}{2\varphi (r_m)^{n-1}}\biggr)^{1/p}\,.
\end{equation}
Note that
\begin{equation*}
\begin{split}
\int_G H(|u_k(x) - (u_k)_G | )\, dx &= \int_G H(|u_k(x)|) \, dx\\ 
&\ge 2\int_{Q_m}H(F(r_m))=2r_m^nH(F(r_m))\,.
\end{split}
\end{equation*}
Hence,  by  \eqref{F}, the $\Delta_2$-condition, and the assumption we have 
\begin{equation*}
\begin{split}
r^n_mH(F(r_m))=
r_m^nH\biggl(\biggl(\frac{{r_m}^{p-1}}{ 2 \varphi (r_m)^{n-1}}\biggr)^{1/p}\biggr)
&\ge r_m^nH\biggl(\frac12 \biggl(\frac{{r_m}^{p-1}}{  \varphi (r_m)^{n-1}}\biggr)^{1/p}\biggr)\\
&\ge \frac1{C_H^{\Delta_2}}r_m^nH\biggl(\biggl(\frac{{r_m}^{p-1}}{ \varphi (r_m)^{n-1}}\biggr)^{1/p}\biggr)
\to\infty,
\end{split}
\end{equation*}
whenever $m \to \infty$.
\end{proof}

Theorem \ref{ineq_fails} implies that condition \eqref{sum} in Theorem \ref{thm:main_application_b}, in the $p=1$ case, is sharp. 

\begin{remark}
Let
$H$ be an Orlicz $\Delta_2$-function which satisfies condition \eqref{eq:H_summable}.
Let us assume that \eqref{sum} holds with $\delta(t) = t^{-\frac{p}{n}}$ in the case $p=1$ i.e.
\[
H\left(h\Big(t^{-\frac{1}{n}}\Big) t + \varphi \Big(t^{-\frac{1}{n}} \Big)^{1-n}  \right)\le
C_H t \quad \text{for all} \quad t\ge 0.
\]
Then we obtain with every function $h$ that
\[
\begin{split}
\lim_{t \to 0^+} t^n H\left(\frac{1}{\varphi (t)^{n-1}}\right)
& \le \lim_{t \to 0^+} t^n H\left( h(t) t^{-n} + \varphi (t)^{1-n} \right)\\
& \le \lim_{t \to \infty} t^{-1} \,  H\left( h\Big(t^{-1/n}\Big)t  +  \varphi \Big(t^{-\frac{1}{n}} \Big)^{1-n}\right) \le C_H.
\end{split}
\]
\end{remark}

Next we study the modified Riesz potential in $\Rn$. 

\begin{theorem}\label{thm:sharpness-Riesz}
Let $\ve, \delta, \beta \ge 0$ and let
$\alpha \in [1, 1+1/(n-1))$.
Let $1\le p<n/(n-\alpha(n-1))$ be given.
Let
$\varphi : (0,\infty )\to \R$
and
$H : [0,\infty )\to \R$ 
be the functions 
\begin{equation*}
\varphi(t) = \frac{t^{\alpha}}{\log ^{\beta }(e+t^{-1})}
\end{equation*}
and
\begin{equation*}H(t)=\biggl(
\frac{t}{\log ^{\beta (n-1)- \delta}(m+t)}
\biggr)^{\frac{np}{\alpha p(n-1)+n(1-p)} + \ve}\,,
\end{equation*}
where $m \ge e$.
If $\ve >0$ or $\delta >0$, then there exists a sequence of functions $(u_k)$ in $ L^p(\Rn)$ such that $\|u_k\|_{L^p(\Rn)}\le C$ and
\[
\int_{\Rn} H\left( \int_{\Rn} \frac{|u_k(z)|}{\varphi(|x - z|)^{n-1}} \, dz \right)\,dx \to \infty \quad \text{as} \quad k \to \infty.
\]
\end{theorem}

\begin{proof}
Let  $A>0$. We will fix $f \in L^p(\Rn)$ later.
By changing the variables we obtain
\begin{align*}
\|A^\frac{n}{p} f(Ax)\|_{L^p(G)}
& = \left(\int_{\Rn} A^n |f(Ax)|^p\,dx \right)^{\frac1p}
=\left(\int_{\Rn}|f(y)|^p\,dy\right)^{\frac1p}\\
&= \|f\|_{L^p(\Rn)}\,.
\end{align*}
On the other hand, by changing the variables,
$Ax=z$ and $Ay=\omega$,
we obtain
\begin{align*}
&\int_{\Rn}  H\biggl(\int_{\Rn} \frac{|A^\frac{n}{p} f(Ax)|}{\varphi (|x-y|)^{n-1}}\,dx\biggr)\,dy\\
&=\int_{\Rn}  H\biggl(\int_{\Rn}  \frac{A^\frac{n}{p} |f(z)|}{A^n\varphi (|\frac{z}{A}-y|)^{n-1}}\,dz\biggr)\,dy\\
&=\int_{\Rn}  A^{-n}H\biggl(A^{\frac{n}{p}-n} \int_{\Rn}  \frac{|f(z)|}{\varphi (|\frac{z}{A}-\frac{\omega}{A}|)^{n-1}}\,dz\biggr)\,d\omega\\
&=\int_{\Rn}  A^{-n}H\biggl(A^{\frac{n}{p}-n} \int_{\Rn}  \frac{|f(z)|}{\varphi ( A^{-1}|z-\omega|)^{n-1}}\,dz\biggr)\,d\omega\,.\\
\end{align*}
Thus, by Fatou's lemma
\[
\begin{split}
&\lim_{A \to \infty}\int_{\Rn}  H\biggl(\int_{\Rn} \frac{|A^\frac{n}{p} f(Ax)|}{\varphi (|x-y|)^{n-1}}\,dx\biggr)\,dy\\ 
&\qquad\ge
\int_{\Rn}  \lim_{A \to \infty} A^{-n}H\biggl(A^{\frac{n}{p}-n} \int_{\Rn}  \frac{|f(z)|}{\varphi ( A^{-1}|z-\omega|)^{n-1}}\,dz\biggr)\,d\omega\,.
\end{split}
\]
Let $f(x)=1$ when $x\in B(0, 2)$ and let $f(x)=0$ otherwise. 
Hence, by substituting
\[
\varphi(t) = \frac{t^{\alpha}}{\log ^{\beta }(e+t^{-1})}
\]
we obtain
that for every $\omega \in B(0,1)$ 
\[
\begin{split}
&A^{-n}H\biggl(A^{\frac{n}{p}-n} \int_{\Rn}  \frac{|f(z)|}{\varphi ( A^{-1}|z-\omega|)^{n-1}}\,dz\biggr) \\
&\quad = A^{-n}H\biggl(A^{\frac{n}{p}-n} \int_{\Rn}  \frac{|f(z)| \log^{\beta(n-1)}(e+A/|z-\omega|)}{A^{-\alpha(n-1)} |z-\omega|^{\alpha(n-1)}}\,dz\biggr) \\
&\quad \ge A^{-n}H\biggl(A^{\frac{n}{p}-n + \alpha(n-1)} \int_{B(\omega,\frac12)}  \frac{\log^{\beta(n-1)}(e+A)}{ |z-\omega|^{\alpha(n-1)}}\,dz \biggr). 
\end{split}
\]
Since $\int_{B(\omega,\frac12)}  \frac{1}{|z-\omega|^{\alpha(n-1)}}\,dz \ge C >0$ for every $\omega \in B(0,1)$ and $H$ is an increasing function, we may estimate
\[
\begin{split}
&\lim_{A \to \infty} A^{-n}H\biggl(A^{\frac{n}{p}-n} \int_{\Rn}  \frac{|f(z)|}{\varphi ( A^{-1}|z-\omega|)^{n-1}}\,dz\biggr)\\
&\quad \ge \lim_{A \to \infty}A^{-n}H\biggl(C A^{\frac{n}{p}-n+ \alpha(n-1)}  \log^{\beta(n-1)}(e+A) \biggr)\,.
\end{split}
\]
By substituting $H$,
\[
H(t)=\biggl(
\frac{t}{\log ^{\beta (n-1)- \delta}(m+t)}
\biggr)^{\frac{np}{\alpha p(n-1)+n(1-p)} + \ve}\,,
\]
we obtain
\[
\begin{split}
&\lim_{A \to \infty} A^{-n}H\biggl(A^{\frac{n}{p}-n} \int_{\Rn}  \frac{|f(z)|}{\varphi ( A^{-1}|z-\omega|)^{n-1}}\,dz\biggr)\\
&=\lim_{A \to \infty}A^{-n} \left( \frac{C A^{\frac{n}{p}-n+ \alpha(n-1)}  \log^{\beta(n-1)}(e+A)}{\log^
{\beta(n-1)- \delta}\left(C A^{\frac{n}{p}-n+ \alpha(n-1)}  \log^{\beta(n-1)}(e+A)\right)} \right)^{\frac{np}{\alpha p(n-1)+n(1-p)} + \ve}\\
&=\lim_{A \to \infty}C A^{\ve(\frac{n}{p}-n+ \alpha(n-1))} \left( \frac{\log^{\beta(n-1)}(e+A)}{\log^
{\beta(n-1)-\delta}\left(C A^{\frac{n}{p}-n+ \alpha(n-1)}  \log^{\beta(n-1)}(e+A)\right)} \right)^{\frac{np}{\alpha p(n-1)+n(1-p)} + \ve}\,.
\end{split}
\]
Note that  $1 \le p <n/(n-\alpha(n-1))$ implies that $\frac{n}{p}-n+ \alpha(n-1)>0$.
If $\ve >0$, then $A^{\ve(\frac{n}{p}-n+ \alpha(n-1))} \to \infty$ as $A \to \infty$, and thus the last limit is infinite for every $\omega \in B(0,1)$. 
If $\delta >0$ (and $\beta \ge 0$), then the term in the brackets tends to infinity as $A\to \infty$,  and thus the last limit is infinite for every $\omega \in B(0,1)$.
Hence, in both cases we obtain
\[
\lim_{A \to \infty}\int_{\Rn}  H\biggl(\int_{\Rn} \frac{|A^\frac{n}{p} f(Ax)|}{\varphi (|x-y|)^{n-1}}\,dx\biggr)\,dy = \infty
\]
and the claim follows.
\end{proof}

\begin{remark}\label{rem:sharpness-Riesz}
By Theorem~\ref{thm:sharpness-Riesz} the exponents $\frac{np}{\alpha p(n-1)+n(1-p)}$ and $\beta(n-1)$ in Corollary~\ref{main_corollary} are the best possible in the sense that our method based on the use of the modified Riesz potential cannot give a better exponent.
\end{remark}


\bibliographystyle{amsalpha}

\begin{thebibliography}{HH}

\bibitem{AF}
R. A. Adams and J. J. F. Fournier,
\emph{Sobolev Spaces} Second edition,
Pure and Applied Mathematics Series, 140. Elsevier/Academic Press, Amsterdam, 2003.

\bibitem{Cianchi}
A. Cianchi,
\textit{A sharp embedding theorem for Orlicz-Sobolev spaces},
Indiana Univ. Math. J.
\textbf{45} (1996), 39--65.



\bibitem{Cianchi1999}
A. Cianchi,
\textit{Strong and weak type inequalities for some classical operators in  Orlicz spaces},
J. London Math. Soc. (2)
\textbf{60} (1999), 187--202.



\bibitem{Cianchi_Stroffolini}
A. Cianchi and B. Stroffolini,
\textit{An extension of Hedberg's convolution inequality and applications},
J. Math. Anal. Appl.
\textbf{227} (1998), 166--186.



\bibitem{Gilbarg-Trudinger}
D.\ Gilbarg and N.\ S. Trudinger,
\emph{Elliptic Partial Differential Equations of Second Order}. Reprint of the 1998 edition. Classics in Mathematics. Springer-Verlag, Berlin, 2001. 





\bibitem{G}
D. J. H. Garling,
\emph{Inequalities, A Journey into Linear Analysis}. Cambridge University Press, Cambridge, 2007.



\bibitem{Haj01}
P.\ Haj\l asz,
\textit{Sobolev inequalities, truncation method, and John domains}, Papers on Analysis: A volume dedicated to Olli Martio on the occasion of his 60th birthday. Edited by  J.\ Heinonen, T.\ Kilpel\"ainen, and P.\ Koskela, 
Report Univ. Jyv\"askyl\"a, \textbf{83}, University of Jyv\"askyl\"a, Jyv\"askyl\"a, 2001, pp. 109--126. 


\bibitem{Hajlasz-Koskela-2}
P.\ Haj\l asz and P.\ Koskela,
\textit{Isoperimetric inequalities and imbedding theorems in irregular domains}, J.\ London Math.\ Soc.\ (2)\textbf{58} (1998), no. 2, 425--450.

\bibitem{Hajlasz-Koskela}
P.\ Haj\l asz and P.\ Koskela,
\textit{Sobolev met Poincar\'e},  Mem.\ Amer.\ Math.\ Soc. \textbf{145} (2000), no.\ 688, x+101 pp.



\bibitem{HH-S}
P.\ Harjulehto and R.\ Hurri-Syrj\"anen,
\textit{An embedding into an Orlicz space for 
$L^1_1$-functions from irregular domains}, 
preprint.





\bibitem{HH-SK}
P.\ Harjulehto, R.\ Hurri-Syrj\"anen, and J.\ Kapulainen,
\textit{An embedding into an Orlicz space for irregular John domains},  Comput. Methods Funct. Theory,
F. W. Gehring Memorial Volume, online March 6, 2014,
DOI 10.1007/s40315-014-0053-3.




\bibitem{HH-SV}
P.\ Harjulehto, R.\ Hurri-Syrj\"anen, and A. V. V\"ah\"akangas,
\textit{On the $(1,p)$-Poincar\'e inequality}, 
Illinois J. Math. \textbf{56} (2012), 905--930.


\bibitem{Hed72}
L.\ I.\ Hedberg,
\textit{On certain convolution inequalities}, Proc.\ Amer.\ Math.\ Soc. \textbf{36} (1972), 505--510.

\bibitem{KM}
T.\ Kilpel\"ainen and J.\ Maly,
\textit{Sobolev inequalities on sets with irregular boundaries},
Z. Anal. Angew. \textbf{19} (2000), 369--380.



\bibitem{Kokilashvili-Krbec}
V.\ Kokilashvili and M.\ Krbec,
\emph{Weighted Inequalities in Lorentz and Orlicz Spaces}. World Scientific. Singapore, 1991. 







\bibitem{Mazya}
V.\ Maz'ya,
\textit{Sobolev Spaces with Applications to Elliptic Partial Differential Equations}, 2nd revised and augmented Edition, A Series of Comprehensive Studies in Mathematics, \textbf{342}, Springer Heidelberg Dordrecht London New York, 2011. 



\bibitem{Mazya-Poborchi}
V.\ Maz'ya and S.\ Poborchi,
\textit{Differentiable Functions on Bad Domains}, World Scientific, Singapore,
1997.


\bibitem{Moser1971}
J. Moser, \textit{A sharp form of an inequality of N. Trudinger's inequality},
Indiana Univ. Math. J.  \textbf{11}
(1971), 1077--1092.






\bibitem{Ohno_Shimomura}
T. Ohno and T. Shimomura,
\textit{Trudinger's inequality for Riesz potentials of functions in Musielak-Orlicz spaces},
Bull. Sci. Math.
\textbf{138} (2014), 225--235.






\bibitem{ONeil1965}
R. O'Neil,
\textit{Fractional integration in Orlicz spaces},
Trans. Amer. Math. Soc.
\textbf{115} (1965), 300--328.





\bibitem{Peetre}
J. Peetre, \textit{Espaces d'interpolation et th\'eor\`eme de Soboleff},
Ann. Inst. Fourier (Grenoble) \textbf{16}
(1966), 279--317.








\bibitem{P}
S. I. Pohozhaev,
\textit{On the imbedding Sobolev theorem for pl=n}, 
Doklady Conference,
Section Math. Moscow Power Inst.(1965), 158--170 (in Russian).



\bibitem{R}
Yu.~G.~Reshetnyak,
\textit{Integral representations of differentiable functions in domains with nonsmooth 
boundary}
(Russian), Sibirsk. Mat. Zh. \textbf{21}(1980), 108--116;
translation in 
Sib. Math. J.\textbf{21} (1981), 833--839.



\bibitem{S}
E.\ M.\ Stein,
\textit{Singular Integrals and Differentiability Properties of Functions}
Princeton Univ. Press, Princeton, New Jersey, 1970.



\bibitem{Strichartz1972}
R. S. Strichartz,
\textit{A note on Trudinger's extension of Sobolev's inequality},
Indiana Univ. Math. J.
\textbf{21} (1972), 841--842.





\bibitem{Torchinsky1976}
A. Torchinsky,
\textit{Interpolation of operators and Orlicz classes},
Studia Math.
\textbf{59} (1976), 177--207.






\bibitem{Tru}
N. S. Trudinger,
\textit{On imbeddings into Orlicz spaces and some applications},
J. Math. Mech. \textbf{17} (1967), 473--483.


\bibitem{Y} 
V. I. Yudovich,
\textit{On some estimates connected with integral operators and
with solutions of elliptic equations} (Russian),
Dokl. Akad. Nauk SSSR \textbf{138} (1961), 805--808; translation in
Sov. Math. Dokl. \textbf{2} (1961), 746--749.
 


\end{thebibliography}

\end{document}